\documentclass[11pt]{amsart}
\usepackage{graphicx}

\usepackage[english]{babel}
\usepackage[latin1]{inputenc}
\usepackage{amsfonts}
\usepackage{amsthm}
\usepackage{amsmath}
\usepackage{amssymb}
\usepackage{amscd}
\usepackage{verbatim}
\usepackage{multicol}
\usepackage{tabularx}
\usepackage{amssymb}
\usepackage[all]{xy}
\usepackage[parfill]{parskip}
\usepackage{mathtools}

\theoremstyle{definition}
\newtheorem{thm}{Theorem}[section]
\newtheorem{prop}[thm]{Proposition}
\newtheorem{cor}[thm]{Corollary}

\newtheorem{ex}{Example}[section]

\DeclareMathOperator{\Aut}{\mathrm{Aut}}

\DeclareMathOperator{\Ad}{\mathrm{Ad}}

\DeclareMathOperator{\ad}{\mathrm{ad}}

\def\ord#1^#2{#1$^{\text{#2}}$}

\def\lie#1{\mathfrak{#1}}

\def\hlie#1{\hat{\mathfrak{#1}}}

\def\uqr#1^#2{\text{$U_q^{#2}(\lie #1)$}}

\def\uqhr#1^#2{\text{$U_q^{#2}(\hlie #1)$}}
\def\us#1^#2{\text{$U_{\xi}^{#2}(\lie #1)$}}
\def\ush#1^#2{\text{$U_{\xi}^{#2}(\hlie #1)$}}
\def\dus#1^#2{\text{$\dot{U}_{\xi}^{#2}(\lie #1)$}}
\def\dush#1^#2{\text{$\dot{U}_{\xi}^{#2}(\hlie #1)$}}

\def\opl_#1^#2{\text{\scriptsize$\bigoplus\limits_{\text{\footnotesize$#1$}}^{\text{\footnotesize$#2$}}$}}
\def\otm_#1^#2{\text{\scriptsize$\bigotimes\limits_{\text{\footnotesize$#1$}}^{\text{\footnotesize$#2$}}$}}

\newcommand{\got}{\mathfrak}

\renewcommand{\thefootnote}

\usepackage{xargs}                 
\usepackage[pdftex,dvipsnames]{xcolor} 

\usepackage[colorinlistoftodos,prependcaption,textsize=tiny]{todonotes}
\newcommandx{\duvida}[2][1=]{\todo[linecolor=red,backgroundcolor=red!25,bordercolor=red,#1]{#2}}
\newcommandx{\comentario}[2][1=]{\todo[linecolor=blue,backgroundcolor=blue!25,bordercolor=blue,#1]{#2}}
\newcommandx{\info}[2][1=]{\todo[linecolor=OliveGreen,backgroundcolor=OliveGreen!25,bordercolor=OliveGreen,#1]{#2}}
\newcommandx{\melhorar}[2][1=]{\todo[linecolor=Plum,backgroundcolor=Plum!25,bordercolor=Plum,#1]{#2}}
\newcommandx{\apagar}[2][1=]{\todo[disable,#1]{#2}}
\allowdisplaybreaks
\begin{document}

%\flushbottom
%%%%%%%%%%%%%%%%%%%%%%%% Topmatter %%%%%%%%%%%%%%%%%%%%%%%%%

\title{Geodesics on adjoint orbits of $SL(n, \mathbb{R})$}
%{\Large Geodesics on the tangent bundle of $SL(\MakeLowercase{n},\mathbb{R})/P_\emptyset.$}}

\author[Rafaela F. do Prado]{Rafaela F. do Prado$^\dag$}
\author[Brian Grajales]{Brian Grajales$^*$}
\author[Lino Grama]{Lino Grama$^\ddag$}

\address{$^\dag$Instituto Federal de Bras\'ilia - IFB, campus Gama, Federal District, Brazil}
\email{rafaela.prado@ifb.edu.br}
\address{$^*$Facultad de Ciencias B\'asicas $-$ Universidad de Pamplona, Departamento de Matem\'aticas. Km 1 v\'ia salida a Bucaramanga, Campus Universitario, Edif. Francisco de Paula Santander, Pamplona, Norte de Santander, Colombia.}
\email{brian.grajales@unipamplona.edu.co}
\address{$^\ddag$IMECC-Unicamp, Departamento de Matem\'{a}tica. Rua S\'{e}rgio Buarque de Holanda 651, Cidade Universit\'{a}ria Zeferino Vaz. 13083-859, Campinas - SP, Brazil.} \email{lino@ime.unicamp.br}

%\email{briang3147@gmail.com, linograma@gmail.com}
\begin{abstract}
  In this paper we study geodesics on adjoint orbits of $SL(n,\mathbb{R})$ equipped with $SO(n)$-invariant metrics (maximal compact subgroup). Our main technique is translate this problem into a geometric problem in the tangent bundle of certain $SO(n)$-flag manifolds and describe the geodesics equations with respect to the Sasaki metric on tangent bundle. We also use tools of Lie Theory in order to obtain some explicit description of families of geodesics. We deal with the case of $SL(2,\mathbb{R})$ in full details.   
\end{abstract}

\maketitle

\section{Introduction}
Let $G$ be a non-compact simple Lie group and let $H_0$ be a regular element in the Cartan sub-algebra of the Lie algebra $\mathfrak{g}$ of $G$. In this paper we will study geodesics on the adjoint orbit $\Ad (G) H_0$.

The adjoint orbit $\Ad (G) H_0$ is the homogeneous space $G/Z(H_0)$, being $Z(H_0)$ the centralizer of $H_0$ and have a nice description from a symplectic geometry point of view: equip $\Ad (G) H_0$ with the Kostant-Kirilov-Souriaux symplectic form and we have a symplectomorphis with $T^* \Ad(K)H_0$, the cotangent bundle of the adjoint orbit of the maximal compact $K$ of $G$ (in fact, $\Ad(K)H_0$ is a {\em generalized flag manifold}) equiped with the canonical symplectic form of the cotangent bundle, see \cite{GGSM} for details. Recent works about the symplectic geometry of non-compact adjoint orbits and applications to the {\em mirror symmetry} program can be found in \cite{BBGGS} and references therein. 

Geodesics on homogeneous spaces is a classical topic in Riemannian geometry with several recent contributions by many authors, see for instance\cite{arv1}, \cite{arv2}, \cite{GGN}, \cite{souris} and references therein. In this paper we will consider Riemannian metrics preserving some symmetries of the homogeneous space, but not all transitive group $G$. More specifically, we will consider Riemannian metrics invariant by the maximal compact subgroup $K\subset G$. In this context we cannot apply directly the standard Lie Theory techniques for the study of geodesics since the Riemannian metrics are not invariant for the whole group $G$. Our strategy is the following: we start with the sequence of identifications 
\begin{equation}
\Ad(G)H_0 \overset{\mathclap{\cong}}{\longrightarrow} T^\ast \Ad(K)H_0 \overset{\mathclap{\cong}}{\longrightarrow} T \Ad(K)H_0,
\end{equation}
and we consider a Sasaki metric on $T \Ad(K)H_0$ induced by an invariant metric on the flag manifold $\Ad(K)H_0$. Finally we obtain a Riemannian metric on the non-compact adjoint orbit $\Ad(G)H_0$ via the pull-back of the Sasaki metric. This resulting metric on $\Ad(G)H_0$ is invariant by the compact subgroup $K$ but not invariant by the whole group $G$. After recall this construction we describe explicitly the geodesics equations on the tangent bundle of $\Ad(K)H_0$ and we provide several examples of {\em horizontal} and {\em oblique} geodesics as an application of our methods. In this paper we will consider the case of adjoint orbtis of $G=SL(n, \mathbb{R})$, associated to the split real forms of the Lie algebra of type $\mathcal{A}$. In the last section of the paper we work-out in details a low-dimensional example on order to study the geodesics on the adjoint orbit of $SL(2,\mathbb{R})$.  

\section{Preliminaries}
\subsection{Adjoint orbits}
Let us consider a non-compact semi-simple Lie algebra $\mathfrak{g}$ (real or complex) and $G$ be a connected Lie group with finite centre and Lie algebra $\mathfrak{g}$. We will fix the Iwasawa decomposition $\mathfrak{g}=\mathfrak{k}\oplus\mathfrak{a}\oplus\mathfrak{n}$ and the corresponding global decomposition $G=KAN$. Let $\Pi$ be a set of simple roots of $\mathfrak{a}$ with a choice of a set of positive roots $\Pi^+$ and $\Sigma\subset \Pi^+$ be the set of simple roots. We denote the corresponding Weyl chamber by $\mathfrak{a}^+$. Let $\Theta\subset \Sigma$ be a choice of simple roots (the choice $\Theta=\emptyset$ is possible). For each choice of $\Theta$ one can associate a parabolic subalgebra $\mathfrak{p}_\Theta$ with parabolic subgroup $P_\Theta$ and a generalized flag manifold $\mathbb{F}_\Theta=G/P_\Theta$. It is worth to point out that in the case $\Theta=\emptyset$, the parabolic subgroup $P_\emptyset$ is a Borel subgroup denoted by $B$ and the corresponding flag manifolds $G/B$ is called {\em maximal flag manifold}. The compact subgroup $K$ also acts transitively on $\mathbb{F}_\Theta$ with isotropy subgroup $K_\Theta=K\cap P_\Theta$, that is, $\mathbb{F}_\Theta=K/K_\Theta$ and $\mathbb{F}_\Theta$ is a compact homogeneous space. Another very well know description of $\mathbb{F}_\Theta$ is as follows: for each $\Theta\subset \Sigma$ one can find $H_\Theta \in \mathrm{cl} \mathfrak{a}^+$ such that
$$
\Theta=\{ \alpha\in \Sigma: \alpha(H_\Theta)=0\}.
$$
In this case one can define the flag manifold $\mathbb{F}_\Theta$ as the (co)adjoint orbit of $K$ throughout $H_\Theta$, that is, $\mathbb{F}_\Theta=\Ad (K)H_\Theta$. 

In this work we are interested for the geometry of the adjoint orbit of the group $G$ instead the group $K$. Let us consider the $Z_\Theta=\{ g\in G:\Ad(g)H_\Theta=H_\Theta \} $ the centralizer in $G$ of the element $H_\Theta$. The next result is a very important ingredient in our work. 
\begin{thm}[\cite{EGSM}]\label{teo-iso}
The adjoint orbit $\Ad(G)\cdot H_\Theta \simeq G/Z_\Theta$ of the element $H_\Theta$ is a $C^\infty$ vector bundle over $\mathbb{F}_\Theta$ that is isomorphic to the cotangent bundle of $T^*\mathbb{F}_\Theta$ in such way that the pullback of the canonical symplectic form on $T^*\mathbb{F}_\Theta$ by this isomorphism is the real Kirillov-Kostant-Souriaux form on the adjoint orbit. 
\end{thm}
We use the Theorem \ref{teo-iso} in order to translate the geometric problems in adjoint orbit of non-compact group $G$ to geometric problems on (co)tangent bunble of flag manifolds. 

\subsection{Invariant Riemannian metrics on flag manifolds $\mathbb{F}_\Theta$}\label{inv-met}
Recall that the flag manifold $\mathbb{F}_{\Theta}=K/K_{\Theta}$ is a reductive homogeneous space, that is, there exists a subspace $\got{m}$ of $\got{k}$ such that 
$$
\got{k}=\got{k}_{\Theta}\oplus \got{m} \ \ \mbox{and} \ \ 
\Ad(k)\got{m}\subseteq \got{m}, \forall k\in K_{\Theta},
$$
where $\mathfrak{k}_\Theta$ is the Lie algebra of $K_\Theta$. One can identify the tangent space $T_{x_{0}}\mathbb{F}_{\Theta}$ with $\got{m}$, where $x_{0}=eK_{\Theta}$ is the origin of the $\mathbb{F}_{\Theta}$ (trivial coset).

%Let us give a description of the tangent space $\got{m}$ in terms of the Lie algebra structure of $\got{g}$ as follows:
%$$
%\got{g}=\got{k}_{\Theta}\oplus\sum_{\beta\in \Pi_{M}}{\got{u}_{\beta}},
%$$
%with $\got{u}_{\beta}=\got{g}\cap(\got{g}_{\beta}\oplus \got{g}_{-\beta})$ and for each root $\beta \in \Pi_{M}$, $\got{u}_{\beta}$ has real dimension two and it is spanned by $A_{\beta}$ and $\sqrt{-1}S_{\beta}$, and we have the following identification
%
%O espaço tangente a uma variedade flag generalizada na origem, visto como um espaço homogêneo de $U$, se identifica com o subespaço de $\got{u}$ gerado por $A_{\beta}$ e $iS_{\beta}$ com $\beta\in \Pi_{M}$, a saber
%$$
%\got{m}=\sum_{\beta\in \Pi_{M}}\got{u}_{\beta}.
%$$

Since flag manifolds are reductive homogeneous space its {\em isotropy representation} is equivalent to the following representation 

%An important ingredient to study invariant tensors on homogeneous space is the {\em isotropy representation}. We will restrict ourselves to the situation of flag manifolds. In this case, since the flag manifolds are {\em reductive} homogeneous spaces it is well know that the isotropy representation is equivalent to the following representation
%$$\left. \Ad(k) \right|_{\got{m}}:\got{m}\longrightarrow\got{m}$$
$$\Ad(k)\mid_{\got{m}}:\got{m}\to \got{m}.$$

In the case of flag manifolds the isotropy representation decomposes $\got{m}$ into irreducible components, that is, 
$$
\got{m}=\got{m}_{1}\oplus\got{m}_{2}\oplus\cdots\oplus\got{m}_{n},
$$
where each component $\got{m}_{i}$ satisfies $\Ad(K_{\Theta})(\got{m}_{i})\subset \got{m}_{i}$. 
We have also that each component $\got{m}_{i}$ is irreducible.%, that is, the only invariant sub-spaces of $\got{m}_{i}$ by $\Ad(K_{\Theta})|_{\got{m}_{i}}$ are the trivial sub-spaces. We will call the sub-spaces $\got{m}_{i}$ by {\it isotropic summands} of the isotropy representation. 

%\begin{obs}
%In the sequel we will omit the symbol $\Theta$ whenever there is no risk %of confusion. We will denote a flag manifold just by $\mathbb{F}=G/K$.
%\end{obs}

%\subsection{Invariant Metrics} 
Let us denote by $Q(\cdot,\cdot)$ the negative of the Cartan-Killing form of $\got{k}$. For each $\Ad(K_\Theta)-$invariant inner product $(\cdot ,\cdot)$ on $\mathfrak{m}$, there exists a unique $Q(\cdot,\cdot)-$self-adjoint, positive operator $P :\mathfrak{m}\to \mathfrak{m}$ commuting with the isotropy representation $\Ad(k)\left|_{\mathfrak{m}}\right.$ for all $k\in K_\Theta$ such that 
$$(X,Y)=Q(P X,Y), \,\, X,Y\in\mathfrak{m}.$$

Therefore an invariant Riemannian metric $g$ on $\mathbb{F}_\Theta$ is completely described by the invariant inner product $(X,Y)$ determined by the linear operator $P$. 

%The vectors $A_{\alpha}$, $\sqrt{-1}S_{\alpha}$, $\alpha\in\Pi$,  are eigenvectors of $\Lambda$ associated to the same eigenvalue $\lambda_\alpha$. 

%The invariant inner product $(X,Y)_{\Lambda}$ admits a natural extension to a symmetric bilinear form on  $\got{m}^{\mathbb{C}}$. On the complexified tangent space we have $\Lambda(X_{\alpha})=\lambda_{\alpha}X_{\alpha}$ with  $\lambda_{\alpha}>0$ and $\lambda_{-\alpha}=\lambda_{\alpha}$. 

%\

In the next sections we will abuse the notation and will denote the invariant metric $g$ on $\mathbb{F}_\Theta$ just by the operator $P$ associated to the invariant inner product. In the case of the matrix associated to $P$ be a diagonal matrix we will call the associated metric by {\em diagonal metric}. 

%{\bf Notation:} In the sequence of this work, we will abuse the notation and will denote the invariant metric $g$ just by the operator $\Lambda$ associated to the invariant inner product. We also denote the invariant metric $g$ just by a $n$-tuple of positive number $(\lambda_1,\ldots, \lambda_n)$ representing the eigenvalues of the operator $\Lambda$ and parametrized by the number of irreducible components.   \qed

\section{The tangent bundle of a homogeneous space with discrete isotropy}
Consider a manifold $M$ with a transitive and effective action of a Lie group $G$ given by
\begin{center}
	$\phi: G\times M\longrightarrow M;$ $\phi(a,p)=a\cdot p.$
\end{center}
Then $M$ can be identified with the quotient $G/H$, where $H$ is the isotropy subgroup of a fixed element $o\in M.$ We will make the following assumptions: {\em we assume that $G/H$ is reductive and that $H$ is a discrete subgroup of $G$}. Let $G^*$ the semi-direct group $\mathfrak{g}\times_\tau G$ given by the adjoint representation 
\begin{center}
	$\tau:=\Ad:G\longrightarrow \Aut(\mathfrak{g})$
\end{center}
(here $\mathfrak{g}$ is considered as an abelian group). Thus, the product in $G^*$ is given by
\begin{center}
	$(X,a)\cdot(Y,b)=(X+\Ad(a)Y,ab)$,
\end{center}
its identity element is $(0,e)$ and $(X,a)^{-1}=(-\Ad(a^{-1})X,a^{-1}),$ for every $(X,a)\in G^*.$ Given $a\in G$ and $p\in M$, denote by $\phi_a:M\longrightarrow M$ and $\varphi_p:G\longrightarrow M$ the maps defined by
\begin{center}
	$\phi_a(x)=a\cdot x$ and $\varphi_p(b)=b\cdot p.$ 
\end{center}
We shall assume that $\varphi_p$ is a submersion for every $p\in M$. Denote by $TM\xrightarrow{\pi}M$ the tangent bundle of $M$ and for every $X\in \mathfrak{g}$, consider the field
\begin{center}
	$X^*(p):=\displaystyle\left.\frac{d}{dt}\exp(tX)\cdot p\ \right|_{t=0}=(d\varphi_p)_e(X).$
\end{center}
\begin{prop} The map
\begin{center}
	 $\tilde{\phi}:G^*\times TM\longrightarrow TM;$ $\tilde{\phi}((X,a),v)=(d\phi_a)_{\pi(v)}(v)+X^*(a\cdot \pi(v))$
\end{center}
is a transitive action. 
\end{prop}
\begin{proof} Observe that
\begin{center}
$\begin{array}{ccl}
\tilde{\phi}((0,e),v)&=& (d\phi_e)_{\pi(v)}(v)+0^*(e\cdot\pi(v))\\
&=& (d\ \text{Id})_{\pi(v)}(v)\\
&=& v
\end{array}$
\end{center}
and 
\begin{center}
	$\begin{array}{ccl}
	\tilde{\phi}((X,a),\tilde{\phi}((Y,b),v))&=& \tilde{\phi}((X,a),(d\phi_b)_{\pi(v)}(v)+Y^*(b\cdot\pi(v)))\\
	&=& (d\phi_a)_{b\cdot\pi(v)}((d\phi_b)_{\pi(v)}(v)+Y^*(b\cdot\pi(v)))+X^*(a\cdot (b\cdot\pi(v)))\\
	&=& (d\phi_{ab})_{\pi(v)}(v)+(d\phi_a)_{b\cdot\pi(v)}(Y^*(b\cdot\pi(v)))+X^*(ab\cdot\pi(v)),\\
	\end{array}$
\end{center}
but 
\begin{center}
	$\begin{array}{ccl}
	\displaystyle(d\phi_a)_{b\cdot\pi(v)}(Y^*(b\cdot\pi(v)))&=&\displaystyle(d\phi_a)_{b\cdot\pi(v)}\left(\left.\frac{d}{dt}\exp(tY)\cdot(b\cdot\pi(v))\right|_{t=0}\right)\\
	\\
	&=& \displaystyle\left.\frac{d}{dt}a\cdot(\exp(tY)b\cdot\pi(v))\right|_{t=0}\\
	\\
	&=&\displaystyle\left.\frac{d}{dt}a\exp(tY)a^{-1}\cdot(ab\cdot\pi(v))\right|_{t=0}\\
	\\
	&=&\displaystyle\left.\frac{d}{dt}\exp(t\Ad(a)Y)\cdot(ab\cdot\pi(v))\right|_{t=0}\\
	\\
	&=&\displaystyle(\Ad(a)Y)^*(ab\cdot\pi(v)),\\
	\end{array}$
\end{center}
so
\begin{center}
	$\begin{array}{ccl}
\tilde{\phi}((X,a),\tilde{\phi}((Y,b),v))&=& (d\phi_{ab})_{\pi(v)}(v)+(d\phi_a)_{b\cdot\pi(v)}(Y^*(b\cdot\pi(v)))+X^*(ab\cdot\pi(v))\\
&=& (d\phi_{ab})_{\pi(v)}(v)+(\Ad(a)Y)^*(ab\cdot\pi(v))+X^*(ab\cdot\pi(v))\\
&=&(d\phi_{ab})_{\pi(v)}(v)+(\Ad(a)Y+X)^*(ab\cdot\pi(v))\\
&=&\tilde{\phi}((X+\Ad(a)Y,ab),v)\\
&=&\tilde{\phi}((X,a)(Y,b),v).
\end{array}$
\end{center}
Thus, $\tilde{\phi}$ is an action. 

Let $v,w\in TM$, since $\phi$ is transitive, there exists $a\in G$ such that $a\cdot\pi(v)=\pi(w)$ and since $\varphi_{\pi(w)}$ is a submersion, there exists $X\in \mathfrak{g}$ such that $(d\varphi_{\pi(w)})_e(X)=w-(d\phi_a)_{\pi(v)}(v),$ hence $\tilde{\phi}((X,a),v)=w$ and $\tilde{\phi}$ is transitive. 
\end{proof}

We can compute the isotropy subgroup $H^*$ of $0\in T_oM$ with respect to the action $\tilde{\phi}$. In fact,
\begin{center}
	$\begin{array}{ccl}
	H^*&=&\{(X,a)\in G^*:\tilde{\phi}((X,a),0)=0\}\\
	   &=&\{(X,a)\in G^*:(d\phi_a)_o(0)+X^*(a\cdot o)=0\}\\
	   &=&\{(X,a)\in G^*:a\cdot o=o\ \text{and}\ X^*(o)=0\}\\
	   &=&\{(X,a)\in G^*:a\cdot o=o \ \text{and}\ X\in\mathfrak{h}\}\\
	   &=&\{(0,a)\in G^*:a\in H\}\ \ \ \ \ \ \ \ \ \ \ \  \text{(since}\ H\ \text{is discrete)}\\
	   &=&\{0\}\times_\tau H.
	\end{array}$
\end{center} 
\textbf{Remark.} If $H$ is not discrete then $H^*=\mathfrak{h}\times_\tau H.$

\begin{cor}
Let us consiser the homogeneous space $G/H$ with $H$ being a discrete subgroup of $G$. Then $TM$ is a homogeneous space diffeomorphic to $G^*/H^*.$
\end{cor}
The diffeomorphism $G^*/H^*\longrightarrow TM$ is given by
\begin{center}
	$(X,a)H^*\longmapsto X^*(a\cdot o).$
\end{center}
\begin{prop}\label{2.3}
	Let us consiser the homogeneous space $G/H$ with $H$ being a discrete subgroup of $G$. Then $G^*/H^*$ is diffeomorphic to $\mathfrak{g}\times(G/H)$. In particular, $M$ is parallelizable.
\end{prop}
\begin{proof} Consider
	\begin{center}
		$\begin{array}{cccc}
		\Omega:&G^*/H^*&\longrightarrow&\mathfrak{g}\times(G/H)\\
		& (X,a)H^*&\longmapsto& (X,aH)
		\end{array}.$
	\end{center}
It is obvious that this map is surjective. Since 
\begin{center}
	$\begin{array}{ccl}
	(X,a)H^*=(Y,b)H^*&\Longleftrightarrow&(X,a)^{-1}(Y,b)\in H^*\\
	&\Longleftrightarrow&(-\Ad(a^{-1})X,a^{-1})(Y,b)\in H^*\\
	&\Longleftrightarrow&(-\Ad(a^{-1})X+\Ad(a^{-1})Y,a^{-1}b)\in H^*\\
	&\Longleftrightarrow&(\Ad(a^{-1})(Y-X),a^{-1}b)\in H^*\\
	&\Longleftrightarrow&\Ad(a^{-1})(Y-X)=0 \text{ and } a^{-1}b\cdot o=o\\
	&\Longleftrightarrow&Y=X \text{ and } aH=bH,
	\end{array}$
\end{center}
then $\Omega$ is well-defined and injective. 
\end{proof}

We recall that the Lie algebra of $G^*=\mathfrak{g}\times_\tau G$ is the semi-direct product $\mathfrak{g}\times_\rho\mathfrak{g},$ where $\rho=\ad$. In this case, the first component $\mathfrak{g}$ is considered as an abelian Lie algebra, so the Lie bracket is given by
\begin{equation}\label{0}
	[(X_1,Y_1),(X_2,Y_2)]=([Y_1,X_2]-[Y_2,X_1],[Y_1,Y_2]).
\end{equation}
\section{Maximal flag manifolds of type $\mathcal{A}$}

In this section we consider the homogeneous space $\mathbb{F}_\emptyset=SL(n,\mathbb{R})/P_\emptyset,$ $n\geq 2$; where $P_\emptyset$ is the parabolic subgroup containing all the upper triangular matrices with determinant equal to 1, that is $\mathbb{F}_\emptyset$ is a maximal flag manifold and $P_\emptyset$ is a Borel subgroup. The special orthogonal group $K=SO(n)$ acts transitively on $\mathbb{F}_\emptyset$ with isotropy $K_\emptyset=S(O(1)\times...\times O(1))$ ($n$ times), so we can identify $\mathbb{F}_\emptyset=K/K_\emptyset=SO(n)/S(O(1)\times...\times O(1)).$ For each $i,j\in\{1,...,n\}$ let $E_{ij}$ be the real $n\times n$ matrix with value equal to 1 in the $(i,j)-$entry and zero elsewhere. Then $\{w_{ij}=E_{ij}-E_{ji}:1\leq j<i\leq n\}$ is a basis for the Lie algebra $\mathfrak{so}(n)$ of $SO(n).$ By Proposition \ref{2.3}, we have that the tangent bundle $T\mathbb{F}_{\emptyset}$ is diffeomorphic to $\mathfrak{so}(n)\times \left(SO(n)/S(O(1)\times...\times O(1))\right)$ via 
\begin{equation}
	\begin{array}{rcc}
	\mathfrak{so}(n)\times \left(\frac{SO(n)}{S(O(1)\times...\times O(1))}\right)&\longleftrightarrow& T\mathbb{F}_{\emptyset}\\
	(X,p\cdot o)&\longrightarrow& X^*(p\cdot o)
	\end{array}.
\end{equation}
Let $g$ an invariant diagonal metric on $\mathbb{F}_{\emptyset}$ and provide $T\mathbb{F}_\emptyset$ with the Sasaki metric $\hat{g}$ induced by $g$. Recall that diagonal metric $g$ on $\mathbb{F}_{\emptyset}$ (in fact, the operator $P$ - see Section \ref{inv-met}) is determined by $\frac{(n-1)n}{2}$ positive numbers $\mu_{ij},$ $1\leq j<i\leq n$ such that
\begin{equation}\label{A3}
g(w_{ij},w_{rs})=\left\{\begin{array}{ll}
\mu_{ij}, & \text{if} \,\, (r,s)=(i,j)\\
\\
0, & \text{otherwise}
\end{array}\right.\ \  (\text{where}\ j<i\ \text{and}\ s<r).
\end{equation}
Although $T\mathbb{F}_\emptyset$ is a homogeneous space and $g$ is invariant, we cannot assure that $\hat{g}$ is invariant and this makes more difficult to obtain an explicit description of $\hat{g}.$ In order to find geodesics on $T\mathbb{F}_\emptyset$, we use the fact that every curve $\gamma$ on $T\mathbb{F}_\emptyset$ can be understood  as a vector field along its projection $\zeta=\pi\circ\gamma$. Next proposition give us necessary and sufficient conditions for a curve $\gamma$ to be geodesic with respect to the Sasaki metric.

\begin{prop}(\cite{DZ}) A curve $\gamma$ on $T\mathbb{F}_\emptyset$ is geodesic with respect to the Sasaki metric $\hat{g}$ if and only if
\begin{equation}\label{3}
\left\{\begin{array}{l}
\displaystyle\nabla_{\dot{\zeta}}\dot{\zeta}=-R(\gamma,\nabla_{\dot{\zeta}}\gamma)\dot{\zeta}\\
\displaystyle\nabla_{\dot{\zeta}}\nabla_{\dot{\zeta}}\gamma=-g(	\displaystyle\nabla_{\dot{\zeta}}\gamma,	\displaystyle\nabla_{\dot{\zeta}}\gamma)\gamma
\end{array},\right.
\end{equation}
where $\zeta(t)=\pi\circ\gamma(t)$ and $\nabla$ is the Levi-Civita connection corresponding to $g.$
\end{prop}

\textbf{Remark.} If $\zeta$ is geodesic on $\mathbb{F}_{\emptyset}$ and $\gamma$ is a parallel vector field along $\zeta$, then we say that $\gamma$ is a \emph{horizontal} geodesic. A geodesic curve $\gamma$ for which $\zeta$ is constant is called \emph{vertical.} A geodesic which is neither horizontal nor vertical is called \emph{oblique.}

Observe that every curve $\gamma$ on $T\mathbb{F}_{\emptyset}$ has the form
\begin{center}
	$\gamma(t)=\displaystyle\left(\sum\limits_{r>s}x_{rs}(t)w_{rs},\alpha(t)K_{\emptyset}\right)=\sum\limits_{r>s}x_{rs}(t)w_{rs}^*(\alpha(t)\cdot o).$
\end{center}
Since 
\begin{center}
	$\begin{array}{ccl}
	X^*(a\cdot o) & = & \displaystyle\left.\frac{d}{dt}\exp(tX)\cdot(a\cdot o)\ \right|_{t=0}\\
	\\
	 & = & \displaystyle\left.\frac{d}{dt}(aa^{-1}\exp(tX)a)\cdot o\ \right|_{t=0}\\
	 \\
	  & = & \displaystyle(d\phi_a)_o\left(\left.\frac{d}{dt}(a^{-1}\exp(tX)a)\cdot o\ \right|_{t=0}\right)\\
	  \\
	  & = & \displaystyle(d\phi_a)_o\left(\left.\frac{d}{dt}\exp(t\Ad(a^{-1})X)\cdot o\ \right|_{t=0}\right)\\
	  \\
	  & = & \displaystyle(d\phi_a)_o\left(\Ad(a^{-1})X\right)^*(o),\\
	  \\
	\end{array}$
\end{center}
for all $X\in\mathfrak{so}(n)$ and $a\in SO(n)$, we have
\begin{equation}
\gamma(t)=\displaystyle\sum\limits_{r>s}x_{rs}(t)(d\phi_{\alpha(t)})_o(\Ad(\alpha(t)^{-1})w_{rs})^*(o).
\end{equation}

\begin{prop}\label{C3.2}The curve $\gamma(t)=\displaystyle\sum\limits_{r>s}x_{rs}(t)(d\phi_{\exp(tw_{ij})})_o(\Ad(\exp(-tw_{ij}))w_{rs})^*(o)$ is a horizontal geodesic if and only if the components $x_{rs}(t)$ satisfy the following conditions:
\begin{equation}\label{system5}
x_{rs}(t) \text{ is constant for } (r,s)=(i,j) \text{ or } r,s\notin\{i,j\}
\end{equation}
\

\begin{equation} \label{system6}
\left\{\begin{array}{l}
\displaystyle x'_{is}(t)\cos(t)+x'_{sj}(t)\sin(t)+\left(\frac{3+\mu_{ij}-\mu_{sj}}{2}\right)(-x_{is}(t)\sin(t)+x_{sj}(t)\cos(t))=0\\
\\
\displaystyle -x'_{is}(t)\sin(t)+x'_{sj}(t)\cos(t)+\left(\frac{3+\mu_{ij}-\mu_{is}}{2}\right)(-x_{is}(t)\cos(t)-x_{sj}(t)\sin(t))=0
\end{array}\right.,j<s<i;
\end{equation}
\

\begin{equation}\label{edo2}
\left\{\begin{array}{l}
\displaystyle x'_{is}(t)\cos(t)-x'_{js}(t)\sin(t)+\left(\frac{3+\mu_{ij}-\mu_{js}}{2}\right)(-x_{is}(t)\sin(t)-x_{js}(t)\cos(t))=0\\
\\
\displaystyle x'_{is}(t)\sin(t)+x'_{js}(t)\cos(t)+\left(\frac{3+\mu_{ij}-\mu_{is}}{2}\right)(x_{is}(t)\cos(t)-x_{js}(t)\sin(t))=0
\end{array}\right., s<j; \ \ \ \ \ \ \ \ 
\end{equation}
\

\begin{equation}\label{edo3}
\left\{\begin{array}{l}
\displaystyle x'_{si}(t)\cos(t)-x'_{sj}(t)\sin(t)+\left(\frac{3+\mu_{ij}-\mu_{sj}}{2}\right)(-x_{si}(t)\sin(t)-x_{sj}(t)\cos(t))=0\\
\\
\displaystyle x'_{si}(t)\sin(t)+x'_{sj}(t)\cos(t)+\left(\frac{3+\mu_{ij}-\mu_{si}}{2}\right)(x_{si}(t)\cos(t)-x_{sj}(t)\sin(t))=0
\end{array}\right., s>i. \ \ \ \ \ \ \ \ \ \ 
\end{equation}
%\end{center}

\end{prop}
\begin{proof} First, observe that
\begin{center}
	$\begin{array}{ccl}
	\gamma(t)&=&\displaystyle\sum\limits_{r>s}x_{rs}(t)(d\phi_{\exp(tw_{ij})})_o(\Ad(\exp(-tw_{ij}))w_{rs})^*(o)\\
	&=&\displaystyle\sum\limits_{r>s}x_{rs}(t)(d\phi_{\exp(tw_{ij})})_o(e^{-t\cdot \ad(w_{ij})}w_{rs})^*(o),
	\end{array}$
\end{center}
and 
\begin{center}
	$\displaystyle\left.\frac{d}{dt}\exp(tw_{ij})K_{\emptyset}=\frac{d}{ds}\exp((t+s)w_{ij})K_{\emptyset}\ \right|_{s=0}=\left.\frac{d}{ds}\exp(tw_{ij})\exp(sw_{ij})K_{\emptyset}\ \right|_{s=0}=(d\phi_{\exp(tw_{ij})})_ow_{ij}^*(o),$
\end{center}
thus
\begin{center}
	$\begin{array}{ccl}
	\displaystyle\nabla_{\dot{\zeta}}\gamma&=&\displaystyle\nabla_{\dot{\zeta}}\left(\sum\limits_{r>s}x_{rs}(t)(d\phi_{\exp(tw_{ij})})_o(e^{-t\cdot \ad(w_{ij})}w_{rs})^*(o)\right)\\
	\\
	&=&\displaystyle\sum\limits_{r>s}x'_{rs}(t)(d\phi_{\exp(tw_{ij})})_o(e^{-t\cdot \ad(w_{ij})}w_{rs})^*(o)+\sum\limits_{r>s}x_{rs}(t)\nabla_{\dot{\zeta}}(d\phi_{\exp(tw_{ij})})_o(e^{-t\cdot \ad(w_{ij})}w_{rs})^*(o)\\
	\\
	&=&\displaystyle\sum\limits_{r>s}x'_{rs}(t)(d\phi_{\exp(tw_{ij})})_o(e^{-t\cdot \ad(w_{ij})}w_{rs})^*(o)\\
	\\
	&&+\displaystyle\sum\limits_{r>s}x_{rs}(t)\nabla_{(d\phi_{\exp(tw_{ij})})_ow_{ij}^*}(d\phi_{\exp(tw_{ij})})_o(e^{-t\cdot \ad(w_{ij})}w_{rs})^*(o)\\
	\\
	&=&\displaystyle(d\phi_{\exp(tw_{ij})})_o\left(\sum\limits_{r>s}x'_{rs}(t)(e^{-t\cdot \ad(w_{ij})}w_{rs})^*(o)+\sum\limits_{r>s}x_{rs}(t)\nabla_{w_{ij}^*}(e^{-t\cdot \ad(w_{ij})}w_{rs})^*(o)\right).
	\end{array}$
\end{center}
Therefore, $\gamma$ is parallel along $\exp(tw_{ij})K_{\emptyset}$ if and only if 
\begin{equation}\label{1}
(d\phi_{\exp(tw_{ij})})_o^{-1}(\nabla_{\dot{\zeta}}\gamma)=\sum\limits_{r>s}x'_{rs}(t)(e^{-t\cdot \ad(w_{ij})}w_{rs})^*(o)+\sum\limits_{r>s}x_{rs}(t)\nabla_{w_{ij}^*}(e^{-t\cdot \ad(w_{ij})}w_{rs})^*(o)=0.
\end{equation} 
We recall that for $X,Y\in\mathfrak{so}(n)$, $\nabla_{X^*}Y^*(o)=(-\frac{1}{2}[X,Y]+U(X,Y))^*(o)$ (\cite{Bes}), where the operator  $U:\mathfrak{so}(n)\times\mathfrak{so}(n)\rightarrow\mathfrak{so}(n)$ is given by the equality 
\begin{equation}\label{2}
	2g(U(X,Y),Z)=g([Z,X],Y)+g([Z,Y],X),\  Z\in\mathfrak{so}(n).
\end{equation}
By using the relations
\begin{align*}
	&e^{-t\cdot \ad(w_{ij})}w_{rs}=w_{rs} \text{ if } (r,s)=(i,j) \text{ or } r,s\notin\{i,j\},\\
	&e^{-t\cdot \ad(w_{ij})}w_{is}=\cos(t)w_{is}+\sin(t)w_{js},\  s\neq j,\ s>i,\\
	&e^{-t\cdot \ad(w_{ij})}w_{js}=\cos(t)w_{js}-\sin(t)w_{is},\ s<j,\\
	&e^{-t\cdot \ad(w_{ij})}w_{si}=\cos(t)w_{si}+\sin(t)w_{sj},\ s>i,\\
	&e^{-t\cdot \ad(w_{ij})}w_{sj}=\cos(t)w_{sj}-\sin(t)w_{si},\ s\neq i,\ s>j,\\
	&U(w_{ij},w_{rs})=0 \text{ if } (r,s)=(i,j) \text{ or } r,s\notin\{i,j\}\\
	&U(w_{ij},w_{is})=\displaystyle\frac{\mu_{ij}-\mu_{is}}{2}w_{js},\ s\neq j,\ s>i,\\
	&U(w_{ij},w_{js})=\displaystyle\frac{\mu_{js}-\mu_{ij}}{2}w_{is},\ s<j,\\
	&U(w_{ij},w_{si})=\displaystyle\frac{\mu_{ij}-\mu_{si}}{2}w_{sj},\ s>i,\\
	&U(w_{ij},w_{sj})=\displaystyle\frac{\mu_{sj}-\mu_{ij}}{2}w_{si},\ s\neq i,\ s>j,\\
	\end{align*}
we have that $(d\phi_{\exp(tw_{ij})})_o^{-1}(\nabla_{\dot{\zeta}}\gamma)$ is equal to
\begin{align*}
	&\displaystyle x'_{ij}(t)w_{ij}^*(o)+\sum\limits_{\begin{subarray}{c}r>s\\r,s\notin\{i,j\}\end{subarray}}x'_{rs}(t)w_{rs}^*(o)\\
	&+\sum\limits_{\begin{subarray}{c}s<i\\s\neq j\end{subarray}}x'_{is}(t)(\cos(t)w_{is}^*+\sin(t)w_{js}^*)(o)+\sum\limits_{s<j}x'_{js}(t)(\cos(t)w_{js}^*-\sin(t)w_{is}^*)(o)\\
	&+\sum\limits_{s>i}x'_{si}(t)(\cos(t)w_{si}^*+\sin(t)w_{sj}^*)(o)+\sum\limits_{\begin{subarray}{c}s>j\\s\neq i\end{subarray}}x'_{sj}(t)(\cos(t)w_{sj}^*-\sin(t)w_{si}^*)(o)\\
	&+\sum\limits_{\begin{subarray}{c}s<i\\s\neq j\end{subarray}}\left\{x_{is}(t)(-\sin(t))w_{is}^*(o)+x_{is}(t)\cos(t)\left(-\frac{1}{2}[w_{ij},w_{is}]+U(w_{ij},w_{is})\right)^*(o)\right\}\\
	&+\sum\limits_{\begin{subarray}{c}s<i\\s\neq j\end{subarray}}\left\{x_{is}(t)\cos(t)w_{js}^*(o)+x_{is}(t)\sin(t)\left(-\frac{1}{2}[w_{ij},w_{js}]+U(w_{ij},w_{js})\right)^*(o)\right\}\\
	&+\sum\limits_{s<j}\left\{x_{js}(t)(-\sin(t))w_{js}^*(o)+x_{js}(t)\cos(t)\left(-\frac{1}{2}[w_{ij},w_{js}]+U(w_{ij},w_{js})\right)^*(o)\right\}\\
	&-\sum\limits_{s<j}\left\{x_{js}(t)\cos(t)w_{is}^*(o)+x_{js}(t)\sin(t)\left(-\frac{1}{2}[w_{ij},w_{is}]+U(w_{ij},w_{is})\right)^*(o)\right\}\\
	&+\sum\limits_{s>i}\left\{x_{si}(t)(-\sin(t))w_{si}^*(o)+x_{si}(t)\cos(t)\left(-\frac{1}{2}[w_{ij},w_{si}]+U(w_{ij},w_{si})\right)^*(o)\right\}\\
	&+\sum\limits_{s<i}\left\{x_{si}(t)\cos(t)w_{sj}^*(o)+x_{si}(t)\sin(t)\left(-\frac{1}{2}[w_{ij},w_{sj}]+U(w_{ij},w_{sj})\right)^*(o)\right\}\\
	&+\sum\limits_{\begin{subarray}{c}s>j\\s\neq i\end{subarray}}\left\{x_{sj}(t)(-\sin(t))w_{sj}^*(o)+x_{sj}(t)\cos(t)\left(-\frac{1}{2}[w_{ij},w_{sj}]+U(w_{ij},w_{sj})\right)^*(o)\right\}\\
	&-\sum\limits_{\begin{subarray}{c}s>j\\s\neq i\end{subarray}}\left\{x_{sj}(t)\cos(t)w_{si}^*(o)+x_{sj}(t)\sin(t)\left(-\frac{1}{2}[w_{ij},w_{si}]+U(w_{ij},w_{si})\right)^*(o)\right\}\\
   &=\displaystyle x'_{ij}(t)w_{ij}^*(o)+\sum\limits_{\begin{subarray}{c}r>s\\r,s\notin\{i,j\}\end{subarray}}x'_{rs}(t)w_{rs}^*(o)\\
	&+\sum\limits_{j<s<i}x'_{is}(t)\cos(t)w_{is}^*(o)+\sum\limits_{s<j}x'_{is}(t)\cos(t)w_{is}^*(o)-\sum\limits_{j<s<i}x'_{is}(t)\sin(t)w_{sj}^*(o)\\
	&+\sum\limits_{s<j}x'_{is}(t)\sin(t)w_{js}^*(o)+\sum\limits_{s<j}x'_{js}(t)\cos(t)w_{js}^*(o)-\sum\limits_{s<j}x'_{js}(t)\sin(t)w_{is}^*(o)\\
	&+\sum\limits_{s>i}x'_{si}(t)\cos(t)w_{si}^*(o)+\sum\limits_{s>i}x'_{si}(t)\sin(t)w_{sj}^*(o)+\sum\limits_{s>i}x'_{sj}(t)\cos(t)w_{sj}^*(o)\\
	&+\sum\limits_{j<s<i}x'_{sj}(t)\cos(t)w_{sj}^*(o)-\sum\limits_{s>i}x'_{sj}(t)\sin(t)w_{si}^*(o)+\sum\limits_{j<s<i}x'_{sj}(t)\sin(t)w_{is}^*(o)\\
	&-\sum\limits_{j<s<i}x_{is}(t)\sin(t)w_{is}^*(o)-\sum\limits_{s<j}x_{is}(t)\sin(t)w_{is}^*(o)-\sum\limits_{j<s<i}x_{is}(t)\cos(t)\left(\frac{1+\mu_{ij}-\mu_{is}}{2}\right)w_{sj}^*(o)\\
	&+\sum\limits_{s<j}x_{is}(t)\cos(t)\left(\frac{1+\mu_{ij}-\mu_{is}}{2}\right)w_{js}^*(o)-\sum\limits_{j<s<i}x_{is}(t)\cos(t)w_{sj}^*(o)+\sum\limits_{s<j}x_{is}(t)\cos(t)w_{js}^*(o)\\
	&-\sum\limits_{j<s<i}x_{is}(t)\sin(t)\left(\frac{1+\mu_{ij}-\mu_{sj}}{2}\right)w_{is}^*(o)-\sum\limits_{s<j}x_{is}(t)\sin(t)\left(\frac{1+\mu_{ij}-\mu_{js}}{2}\right)w_{is}^*(o)\\
	&-\sum\limits_{s<j}x_{js}(t)\sin(t)w_{js}^*(o)-\sum\limits_{s<j}x_{js}(t)\cos(t)\left(\frac{1+\mu_{ij}-\mu_{js}}{2}\right)w_{is}^*(o)-\sum\limits_{s<j}x_{js}(t)\cos(t)w_{is}^*(o)\\
	&-\sum\limits_{s<j}x_{js}(t)\sin(t)\left(\frac{1+\mu_{ij}-\mu_{is}}{2}\right)w_{js}^*(o)-\sum\limits_{s>i}x_{si}(t)w_{si}^*(o)+\sum\limits_{s>i}x_{si}(t)\cos(t)\left(\frac{1+\mu_{ij}-\mu_{si}}{2}\right)w_{sj}^*(o)\\
	&+\sum\limits_{s>i}x_{si}(t)\cos(t)w_{sj}^*(o)-\sum\limits_{s>i}x_{si}(t)\sin(t)\left(\frac{1+\mu_{ij}-\mu_{sj}}{2}\right)w_{si}^*(o)-\sum\limits_{j<s<i}x_{sj}(t)\sin(t)w_{sj}^*(o)\\
	&-\sum\limits_{s>i}x_{sj}(t)\sin(t)w_{sj}^*(o)-\sum\limits_{s>i}x_{sj}(t)\cos(t)\left(\frac{1+\mu_{ij}-\mu_{sj}}{2}\right)w_{si}^*(o)\\
	&+\sum\limits_{j<s<i}x_{sj}(t)\cos(t)\left(\frac{1+\mu_{ij}-\mu_{sj}}{2}\right)w_{is}^*(o)+\sum\limits_{j<s<i}x_{sj}(t)\cos(t)w_{is}^*(o)-\sum\limits_{s>i}x_{sj}(t)\cos(t)w_{si}^*(o)\\
	&-\sum\limits_{s>i}x_{sj}(t)\sin(t)\left(\frac{1+\mu_{ij}-\mu_{si}}{2}\right)w_{sj}^*(o)-\sum\limits_{j<s<i}x_{sj}(t)\sin(t)\left(\frac{1+\mu_{ij}-\mu_{is}}{2}\right)w_{sj}^*(o)\\
	&=\displaystyle x'_{ij}(t)w_{ij}^*(o)+\sum\limits_{\begin{subarray}{c}r>s\\r,s\notin\{i,j\}\end{subarray}}x'_{rs}(t)w_{rs}^*(o)\\
	&+\sum\limits_{j<s<i}\left\{x'_{is}(t)\cos(t)+x'_{sj}(t)\sin(t)+\left(\frac{3+\mu_{ij}-\mu_{sj}}{2}\right)(-x_{is}(t)\sin(t)+x_{sj}(t)\cos(t))\right\}w_{is}^*(o)\\
	&+\sum\limits_{j<s<i}\left\{-x'_{is}(t)\sin(t)+x'_{sj}(t)\cos(t)+\left(\frac{3+\mu_{ij}-\mu_{is}}{2}\right)(-x_{is}(t)\cos(t)-x_{sj}(t)\sin(t))\right\}w_{sj}^*(o)\\
	&+\sum\limits_{s<j}\left\{x'_{is}(t)\cos(t)-x'_{js}(t)\sin(t)+\left(\frac{3+\mu_{ij}-\mu_{js}}{2}\right)(-x_{is}(t)\sin(t)-x_{js}(t)\cos(t))\right\}w_{is}^*(o)\\
	&+\sum\limits_{s<j}\left\{x'_{is}(t)\sin(t)+x'_{js}(t)\cos(t)+\left(\frac{3+\mu_{ij}-\mu_{is}}{2}\right)(x_{is}(t)\cos(t)-x_{js}(t)\sin(t))\right\}w_{js}^*(o)\\
	&+\sum\limits_{s>i}\left\{x'_{si}(t)\cos(t)-x'_{sj}(t)\sin(t)+\left(\frac{3+\mu_{ij}-\mu_{sj}}{2}\right)(-x_{si}(t)\sin(t)-x_{sj}(t)\cos(t))\right\}w_{si}^*(o)\\
	&+\sum\limits_{s>i}\left\{x'_{si}(t)\sin(t)+x'_{sj}(t)\cos(t)+\left(\frac{3+\mu_{ij}-\mu_{si}}{2}\right)(x_{si}(t)\cos(t)-x_{sj}(t)\sin(t))\right\}w_{sj}^*(o),
\end{align*}
so we have the result.
\end{proof}

\begin{ex}[Family of horizontal geodesics I]
Let us fix $w_{ij}\in T_o\mathbb{F}$ and fix $s$ such that $j<s<i$. We will compute the functions $x_{is}$ and $x_{sj}$ by solving the system of ODEs (\ref{system6}) and consequently determine the components of the vector field along the curve $\alpha(t)=\exp(t w_{ij})$, with $\alpha^\prime (0)=aw_{is}+bw_{sj}$. By solving the system (\ref{system6}) we have
$$
x_{is}(t)=a \cos\left( \frac{3+\mu_{is}-\mu_{sj}}{2} t \right) - b\sin \left( \frac{3+\mu_{is}-\mu_{sj}}{2} t \right), 
$$
$$
x_{sj}(t)=b \cos\left( \frac{3+\mu_{is}-\mu_{sj}}{2} t \right) + a\sin \left( \frac{3+\mu_{is}-\mu_{sj}}{2} t \right),
$$
where $\mu_{is}$ and $\mu_{sj}$ are the components of the metric. It is worth to point out that the trajectory of solution $(x_{is}(t),x_{sj}(t))$ is contained in an ellipse in the plane spanned by $w_{is}$ and $w_{sj}$. In the next Section we will provide an explicitly low-dimensional example of this phenomena. Similar computation holds for Equations (\ref{edo2}) and (\ref{edo3}). 
\end{ex}

\begin{ex}[Family of horizontal geodesics II]
The solution for the Equation (\ref{system5}) has a nice solution in therm of the Lie algebra structure of type $\mathcal{A}_\ell$. Let us assume that the vector $w_{ij}$ is the generator of the root space associated to the root $\alpha$ and $w_{rs}$ is associated to the root $\beta$ (after an appropriated normalization in both cases). If we impose the condition that $\alpha$ and $\beta$ are {\em orthogonal} roots, that is $\alpha \pm \beta$ is not a root, this mean in the Lie algebra of type $\mathcal{A}$ that $r,s\notin\{i,j\}$. Therefore if we consider the vector $v=w_{ij}+w_{rs}$ associate to orthogonal roots and take the parallel transport of $v$ along the curve $\exp(t w_{ij})$ we will produce an horizontal geodesics.

\end{ex}
\begin{prop}
	The curve
\begin{center} $\gamma(t)=x_{ij}(t)(d\phi_{\exp(tw_{ij})})_ow_{ij}^*(o)+\sum\limits_{r,s\notin\{i,j\}}x_{rs}(t)(d\phi_{\exp(tw_{ij})})_ow_{rs}^*(o)$
\end{center}
where $(x_{uv})$ satisfies the system of differential equations
\begin{equation}\label{8}
\displaystyle x_{uv}''=-\left(\mu_{ij}(x_{ij}')^2+\sum\limits_{r,s\notin\{i,j\}}\mu_{rs}(x_{rs}')^2\right)x_{uv},\ (u,v)=(i,j)\ \text{or}\ u,v\notin\{i,j\}
\end{equation}
is an oblique geodesic.
\end{prop}
\begin{proof} Let $\zeta=\exp(tw_{ij})\cdot o$. As in the proof of Proposition 2.2 we have that 
\begin{center}
	$\displaystyle \nabla_{\dot{\zeta}}\gamma=(d\phi_{\exp(tw_{ij})})_o\left(x_{ij}'(t)w_{ij}^*(o)+\sum\limits_{r,s\notin\{i,j\}}x_{rs}'(t)w_{rs}^*(o)\right).$
\end{center}
Since $[\gamma,\nabla_{\dot{\zeta}}\gamma]=\sum\limits_{r,s\notin\{i,j\}}f_{rs}(t)w_{rs}^*(o)$ for some functions $f_{rs}$, then 
\begin{center}
	$R(\gamma,\nabla_{\dot{\zeta}}\gamma)=\nabla_{\gamma}\nabla_{\nabla_{\dot{\zeta}}\gamma}\dot{\zeta}-\nabla_{\nabla_{\dot{\zeta}}\gamma}\nabla_{\gamma}\dot{\zeta}-\nabla_{[\gamma,\nabla_{\dot{\zeta}}\gamma]}\dot{\zeta}=0.$
\end{center}
Observe that 
\begin{center}
$\nabla_{\dot{\zeta}}\nabla_{\dot{\zeta}}\gamma=(d\phi_{\exp(tw_{ij})})_o\left(x_{ij}''(t)w_{ij}^*(o)+\sum\limits_{r,s\notin\{i,j\}}x_{rs}''(t)w_{rs}^*(o)\right)$
\end{center}
and 
\begin{center}
	$g(\nabla_{\dot{\zeta}}\gamma,\nabla_{\dot{\zeta}}\gamma)=\left(\mu_{ij}(x_{ij}')^2+\sum\limits_{r,s\notin\{i,j\}}\mu_{rs}(x_{rs}')^2\right)$,
\end{center}
therefore, by \eqref{8} we have that $\nabla_{\dot{\zeta}}\nabla_{\dot{\zeta}}\gamma=-g(\nabla_{\dot{\zeta}}\gamma,\nabla_{\dot{\zeta}}\gamma)\gamma$. Hence $\gamma$ is a geodesic.
\end{proof}

\begin{ex}
Fix $w_{ij}\in T_o\mathbb{F}$. We are looking for the component $x_{ij}$ of the vector field along the curve $\alpha(t)=\exp (tw_{ij})$ that are oblique geodesics. In this case the differential equation reads 
\begin{equation}
x_{ij}^{\prime\prime}=-\mu_{ij}x_{ij}(x_{ij}^\prime)^2,
\end{equation}
where $\mu_{ij}$ is the parameter of the metric in the direction $w_{ij}$. Dividing by $x_{ij}^\prime$ we obtain an exact differential equation
\begin{equation}\label{eq-dif01}
\ln |x_{ij}^\prime|=\int{- \mu_{ij}x_{ij}\,\, dx_{ij}} + C_1.  
\end{equation}
Solving the Equation \ref{eq-dif01} for $x_{ij}^\prime$ we obtain 
\begin{equation}
x_{ij}(t)=\frac{\sqrt{2}}{\sqrt{\mu_{ij}}}\sum_{k=0}^\infty{\frac{(-1)^k c_k}{2k+1}\left(\frac{\sqrt{2\mu_{ij}}}{2} e^{C_1}t+C_2\right)},
\end{equation}
where $C_1=x(0)$, $C_2=x^\prime (0)$ (determined by initial velocity $w_{ij}$) and the numbers $c_k$ are given by 
$$
c_k=\sum_{m=0}^{k-1}{\frac{c_m c_{k-1-m}}{(m+1)(2m+1)}}.
$$
In the next section we will provide low-dimensional explicit examples of this kind of solutions in terms of elementary functions. 
\end{ex}
\section{The tangent bundle of $\mathbb{R}P^1$}

For $n=2$, the flag $\mathbb{F}_\emptyset=SO(2)/S(O(1)\times O(1))$ is diffeomorphic to the real projective space $\mathbb{R}P^1$ which, in turn, is diffeomorphic to $\mathbb{S}^1.$ In this case, we have an identification between $T\mathbb{S}^1$ and the one-sheeted hyperboloid $\mathcal{H}=\{(x,y,z)\in\mathbb{R}^3:\ x^2+y^2-z^2=1\}$; this will help us to find explicitly the geodesics on the tanget bundle. We provide $SO(2)/S(O(1)\times O(1))$ with the invariant metric $g$ given by  $g(w_{21},w_{21})=\mu$, for some $\mu>0$ (see \eqref{A3}).

\begin{prop}The map 
	\begin{equation*}
	\begin{array}{rccc}
	F:& \displaystyle\frac{SO(2)}{S(O(1)\times O(1))} & \longrightarrow & \mathbb{S}^1
	\end{array};\ \left(\begin{array}{cc}
	a & b\\
	c & d
	\end{array}\right)S(O(1)\times O(1)) \longmapsto (a^2-c^2,2ac)
	\end{equation*}
is a diffeomorphism such that $(F^{-1})^*g=h$ where $h=\displaystyle\left.\frac{\mu}{4}(dx^2+dy^2)\right|_{\mathbb{S}^1}$.
\end{prop}
\begin{proof} First, observe that $F$ is well-defined since
	\begin{center}
		$\begin{array}{ccl}
		\left(\begin{array}{cc}
		a & b\\
		c & d
		\end{array}\right)S(O(1)\times O(1))=\left(\begin{array}{cc}
		a' & b'\\
		c' & d'
		\end{array}\right)S(O(1)\times O(1))&\Longrightarrow& \left(\begin{array}{cc}
		a & b\\
		c & d
		\end{array}\right)=\pm \left(\begin{array}{cc}
		a' & b'\\
		c' & d'
		\end{array}\right)\\
		\\
		&\Longrightarrow& (a^2-c^2,2ac)=((a')^2-(c')^2,2a'c')
		\end{array}$
	\end{center}                                                                              
and $(a^2-c^2)^2+(2ac)^2=(a^2+c^2)^2=1$ since $(a,c)^T$ is a column of an orthogonal matrix. It is easy to verify that the map 
\begin{equation*}
G:\mathbb{S}^1\longrightarrow SO(2)/S(O(1)\times O(1))
\end{equation*}
defined by
\begin{equation*} 
G(x,y)=\left\{\begin{array}{ll} \left(\begin{array}{cc}
\frac{x+1}{\sqrt{(x+1)^2+y^2}} & -\frac{y}{\sqrt{(x+1)^2+y^2}}\\ 
\frac{y}{\sqrt{(x+1)^2+y^2}} & \frac{x+1}{\sqrt{(x+1)^2+y^2}}
\end{array}\right)S(O(1)\times O(1)),& \text{if}\ (x,y)\neq(-1,0)\\
\\
\left(\begin{array}{cc}
0 & -1\\ 
1 & 0
\end{array}\right)S(O(1)\times O(1)),& \text{if}\ (x,y)=(-1,0)\\
\end{array}\right.
\end{equation*}
the inverse of $F$. Smoothness of $G$ follows from the fact that it can be rewritten as
\begin{equation*} 
G(x,y)=\left\{\begin{array}{ll} \left(\begin{array}{cc}
\frac{y}{\sqrt{(1-x)^2+y^2}} & -\frac{1-x}{\sqrt{(1-x)^2+y^2}}\\ 
\frac{1-x}{\sqrt{(1-x)^2+y^2}} & \frac{y}{\sqrt{(1-x)^2+y^2}}
\end{array}\right)S(O(1)\times O(1)),& \text{if}\ (x,y)\neq(1,0)\\
\\
\left(\begin{array}{cc}
1 & 0\\ 
0 & 1
\end{array}\right)S(O(1)\times O(1)),& \text{if}\ (x,y)=(1,0)\\
\end{array}\right..
\end{equation*}
For the second part, note that given $Q=\left(\begin{array}{cc}
a & b\\ 
c & d
\end{array}\right)\in SO(2),$ we have
\begin{equation*}
\begin{array}{ccl}
(dF)_{Q\cdot o}(w_{21}^*(Q\cdot o)) &=&\displaystyle(dF)_{Q\cdot o}\left(\left.\frac{d}{dt}\exp(tw_{21})\cdot(Q\cdot o)\right|_{t=0}\right)\\
\\
&=&\displaystyle(dF)_{Q\cdot o}\left(\left.\frac{d}{dt}\left(\begin{array}{cc}
\cos(t) & -\sin(t)\\ 
\sin(t) & \cos(t)
\end{array}\right)\left(\begin{array}{cc}
a & b\\ 
c & d
\end{array}\right)\cdot o\ \right|_{t=0}\right)\\
\\
&=&\displaystyle\left.\frac{d}{dt}F\left(\left(\begin{array}{cc}
a\cdot \cos(t)-c\cdot \sin(t) & b\cdot \cos(t)-d\cdot \sin(t)\\ 
a\cdot \sin(t)+c\cdot \cos(t) & b\cdot \sin(t)+d\cdot \cos(t)
\end{array}\right)\cdot o\right)\right|_{t=0}\\
\\
&=&\displaystyle-4ac\left.\frac{\partial}{\partial x}\right|_{F(Q\cdot o)}+2(a^2-c^2)\left.\frac{\partial}{\partial y}\right|_{F(Q\cdot o)},\\
\end{array}
\end{equation*} 
therefore

\begin{equation*}
\begin{array}{ccl}
(F^*h)_{Q\cdot o}(w_{21}^*(Q\cdot o),w_{21}^*(Q\cdot o)) &=&\displaystyle h_{F(Q\cdot o)}((dF)_{Q\cdot o}(w_{21}^*(Q\cdot o)),(dF)_{Q\cdot o}(w_{21}^*(Q\cdot o))\\
\\
&=&\displaystyle \frac{\mu}{4}((-4ac)^2+(2(a^2-c^2))^2)\\
\\
&=&\displaystyle \frac{\mu}{4}(4(a^2+c^2)^2)\\
\\
&=&\mu,
\end{array}
\end{equation*}
hence, $F^*h=g$. 
\end{proof}

\begin{prop} The one-sheeted hyperboloid $\mathcal{H}=\{(x,y,z)\in\mathbb{R}^3:\ x^2+y^2-z^2=1\}$ is a vector bundle over $\mathbb{S}^1$ which is isomorphic to $T\mathbb{S}^1$.
\end{prop}
\begin{proof} The projection map of $\mathcal{H}$ is given by
	\begin{equation*}	P:\mathcal{H}\longrightarrow \mathbb{S}^1;\ \ 	
	(x,y,z)\longmapsto\displaystyle \left(\frac{x+yz}{1+z^2},\frac{y-xz}{1+z^2}\right).
	\end{equation*}
	
For each $(r,s)\in\mathbb{S}^1$ we have that $P^{-1}\{(r,s)\}=(r,s,0)+span\{(-s,r,1)\},$ so $P$ is surjective and $P^{-1}\{(r,s)\}$ has a vector space structure. Also, $P$ is a submersion because 
\begin{equation*}
(dP)_{(x,y,z)}\displaystyle\left(-y,x,0\right)=\left(\frac{-y+xz}{1+z^2},\frac{x+yz}{1+z^2}\right)\neq 0.
\end{equation*}

To show that $\mathcal{H}$ is isomorphic to $T\mathbb{S}^1$ we consider the map
	\begin{equation*}
	\begin{array}{lcc}
	\mathcal{H}&\xrightarrow{\ \ \ \displaystyle\Psi\ \ \ }& T\mathbb{S}^1\\
	\Bigg\downarrow{P}& &\Bigg\downarrow{\pi}\\
	\mathbb{S}^1&\xrightarrow{\ \ \ \displaystyle\text{Id}\ \ \ }&\mathbb{S}^1
	\end{array}
\end{equation*}
defined by $\Psi(x,y,z)=\displaystyle \left(\frac{z(-y+xz)}{1+z^2},\frac{z(x+yz)}{1+z^2}\right)$ and observe that

\begin{equation*}
\Psi\big|_{P^{-1}\{(r,s)\}}(r-ts,s+tr,t)=\displaystyle \left(-ts,tr\right),
\end{equation*}

so $\Psi\big|_{P^{-1}\{(r,s)\}}:P^{-1}\{(r,s)\}\longrightarrow\pi^{-1}\{(r,s)\}$ is an isomorphism. 
\end{proof}
\textbf{Remark.} We use the fact that $\mathcal{H}$ is the union of lines (also called a \textit{ruled surface})  to see it as a vector bundle over $\mathbb{S}^1$, which is identified with the set $\{(x,y,z)\in\mathbb{R}^3:\ x^2+y^2=1,\ z=0\}$. The map $\Psi$ associates to every point $(x,y,z)\in\mathcal{H}$ the orthogonal projection (with respect to the euclidean metric in $\mathbb{R}^3$) over the tangent line to $\mathbb{S}^1$ at $(P(x,y,z),0)$ (see Figure 1.)
\begin{center}\label{Figure1}
	\includegraphics[width=80mm,scale=1.5]{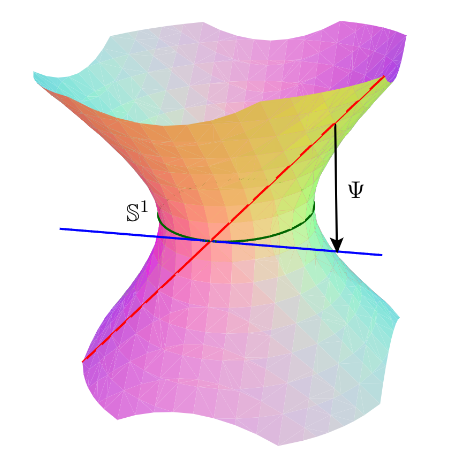}\\
	\small{Figure 1.}
\end{center}

We shall use the isomorphism $\Psi$ to describe the Sasaki metric $\hat{g}$ on $\mathcal{H}$. Fix $q_0=(x_0,y_0,z_0)\in\mathcal{H}$, it is a straightforward calculation to verify that if
\begin{equation*}
\gamma_1(t)=\displaystyle\left(\frac{x_0+y_0z_0}{1+z_0^2},\frac{y_0-x_0z_0}{1+z_0^2},0\right)+(t+z_0)\left(\frac{-y_0+x_0z_0}{1+z_0^2},\frac{x_0+y_0z_0}{1+z_0^2},1\right)
\end{equation*}
then $(P\circ\gamma_1)(t)$ is constant, therefore, $\gamma_1(t)$ is a vertical curve. Also
\begin{center}
	$\begin{array}{ccl}
    \hat{g}_{q_0}(\gamma_1'(0),\gamma_1'(0))&=&h_{P(q_0)}((\Psi\circ\gamma_1)'(0),(\Psi\circ\gamma_1)'(0))\\
    \\
    &=&\displaystyle h_{P(q_0)}\left(\left(\frac{-y_0+x_0z_0}{1+z_0^2},\frac{x_0+y_0z_0}{1+z_0^2}\right),\left(\frac{-y_0+x_0z_0}{1+z_0^2},\frac{x_0+y_0z_0}{1+z_0^2}\right)\right)\\
    \\
    &=&\displaystyle\frac{\mu}{4}\left(\left(\frac{-y_0+x_0z_0}{1+z_0^2}\right)^2+\left(\frac{x_0+y_0z_0}{1+z_0^2}\right)^2\right)\\
    \\
    &=&\displaystyle\frac{\mu}{4}.
	\end{array}$
\end{center}
On the other hand, Proposition \ref{C3.2} implies that given $\xi\in\mathbb{R}$ and $\alpha(t)=\left(\begin{array}{cc}
\cos(t+t_0)&-\sin(t+t_0)\\
\sin(t+t_0)&\cos(t+t_0)\\
\end{array}\right),$ the curve $\eta:
t\longmapsto\xi w_{21}^*(\alpha(t)\cdot o)$ is horizontal in $T\mathbb{F}_\emptyset$ with respect to $\hat{g}$ and it is identified with
\begin{center} $\gamma_2(t)=\Psi^{-1}((dF)_{\alpha(t)\cdot o}(\eta(t)))=(\cos(2(t+t_0))-2\xi \sin(2(t+t_0)),\sin(2(t+t_0))+2\xi \cos(2(t+t_0)),2\xi)\in\mathcal{H}.$ 
\end{center}
By taking $\xi=\frac{z_0}{2}$ and $t_0$ such that 
\begin{center}
	$\left\{\begin{array}{l}
	\cos(2t_0)-z_0\sin(2t_0)=x_0\\
	\sin(2t_0)+z_0\cos(2t_0)=y_0\\
	\end{array}\right.$
\end{center}
we have that $\gamma_2(0)=q_0$ and
\begin{center}
	$\begin{array}{ccl}
	\hat{g}_{q_0}(\gamma_2'(0),\gamma_2'(0))&=&\hat{g}_{\eta(0)}(\eta'(0),\eta'(0))\\
	\\
	&=&\displaystyle g_{\alpha(0)\cdot o}\left(\left.\frac{d}{dt}\alpha(t)\cdot o\ \right|_{t=0},\left.\frac{d}{dt}\alpha(t)\cdot o\ \right|_{t=0}\right)\\
	\\
	&=&\displaystyle g_{\alpha(0)\cdot o}\left(\left.\frac{d}{dt}\exp(tw_{21})\cdot(\alpha(0)\cdot o)\ \right|_{t=0},\left.\frac{d}{dt}\exp(tw_{21})\cdot(\alpha(0)\cdot o) \right|_{t=0}\right)\\
	\\
	&=&\displaystyle g(w_{21},w_{21})\\
	\\
	&=&\mu.
	\end{array}$
\end{center}
Summarizing, we have that given $q_0=(x_0,y_0,z_0)\in\mathcal{H},$
\begin{equation*}
\begin{array}{ccl}
T_{q_0}\mathcal{H}&=&span\{\gamma_1'(0),\gamma_2'(0)\}\\
\\
&=&\displaystyle span\left\{\left(\frac{-y_0+x_0z_0}{1+z_0^2},\frac{x_0+y_0z_0}{1+z_0^2},1\right),\left(-2y_0,2x_0,0\right)\right\}\\
\\
&=&\displaystyle span\left\{\left(\frac{-y_0+x_0z_0}{1+z_0^2},\frac{x_0+y_0z_0}{1+z_0^2},1\right),\left(-y_0,x_0,0\right)\right\}
\end{array}
\end{equation*} 
where the generators satisfy 
\begin{equation}\label{8}
\begin{array}{l}
i)\ \displaystyle \hat{g}_{q_0}\left(\left(\frac{-y_0+x_0z_0}{1+z_0^2},\frac{x_0+y_0z_0}{1+z_0^2},1\right),\left(\frac{-y_0+x_0z_0}{1+z_0^2},\frac{x_0+y_0z_0}{1+z_0^2},1\right)\right)=\frac{\mu}{4},\\
\\
ii)\ \displaystyle\hat{g}_{q_0}(\left(-y_0,x_0,0\right),\left(-y_0,x_0,0\right))=\frac{\mu}{4}\ \text{and}\\
\\
iii)\ \displaystyle\hat{g}_{q_0}\left(\left(\frac{-y_0+x_0z_0}{1+z_0^2},\frac{x_0+y_0z_0}{1+z_0^2},1\right),\left(-y_0,x_0,0\right)\right)=0.\end{array}
\end{equation}

Now, consider the parametrization
\begin{equation*}
\left\{\begin{array}{l}
x=\cos(u)-v\sin(u)\\
y=\sin(u)+v\cos(u)\\
z=v
\end{array}\right.,\hspace{1cm} (u,v)\in(-\pi,\pi)\times\mathbb{R}
\end{equation*}
of $\mathcal{H}.$ For each $(u,v)\in(-\pi,\pi)\times\mathbb{R}$ we have

\begin{equation*}
(dX)\left(\displaystyle\frac{\partial}{\partial u}\right)=-y\frac{\partial}{\partial x}+x\frac{\partial}{\partial y}
\end{equation*}
and
\begin{equation*}
(dX)\left(\displaystyle\frac{\partial}{\partial v}\right)=\left(\frac{-y+xz}{1+z^2}\right)\frac{\partial}{\partial x}+\left(\frac{x+yz}{1+z^2}\right)\frac{\partial}{\partial y}+\frac{\partial}{\partial z},
\end{equation*}

(where $X(u,v)=(x,y,z)$), so, by \eqref{8} we obtain that $X^*\hat{g}=\displaystyle \frac{\mu}{4}(du^2+dv^2)$. Then, geodesics on $(\mathcal{H},\hat{g})$ are direct images under $X$ of lines
\begin{equation}\label{9}
au+bv=c,\ \ (u,v)\in(-\pi,\pi)\times\mathbb{R},\ a,b,c\in\mathbb{R}.
\end{equation}
We distinguish three cases:

\textit{Case 1.} $a=0$ $b\neq0$:\\

In this case, equation \eqref{9} is equivalent to

$$v=\displaystyle\frac{c}{b},$$
	
whose image by $X$ is a circle in $\mathcal{H}$ of radius $\sqrt{1+\left(\frac{c}{b}\right)^2}$ contained in the plane $z=\frac{c}{b}$. These are horizontal geodesics.

\begin{center}
\includegraphics{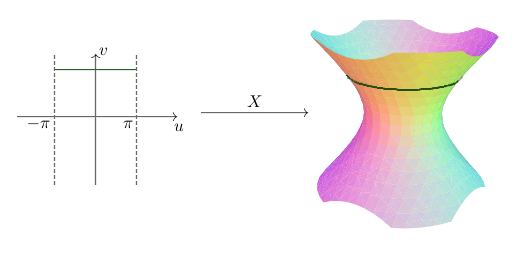} 
\end{center}

\textit{Case 2.} $a\neq0$, $b=0$:\\

Equation \eqref{9} becomes 

$$u=\displaystyle\frac{c}{a}.$$

Its image by $X$ is the line
\begin{center}
	$\textbf{r}(t)=\left(\cos\left(\frac{c}{a}\right),\sin\left(\frac{c}{a}\right),0\right)+t\left(-\sin\left(\frac{c}{a}\right),\cos\left(\frac{c}{a}\right),1\right),$ \ \ $t\in\mathbb{R}.$ 
\end{center}
These are vertical geodesics.

\textit{Case 3.} $a\neq0$, $b\neq0$:\\

When $a$ and $b$ are simultaneously non-zero we obtain a ``spiral" contained in $\mathcal{H}$, it can be parametrized by 
\begin{center}
	$\textbf{x}(t)=\displaystyle\left(\cos\left(t\right)-\frac{c-at}{b}\sin(t),\sin\left(t\right)+\frac{c-at}{b}\cos(t),\frac{c-at}{b}\right),$ \ \ $t\in\mathbb{R}.$ 
\end{center}
Such a geodesic is \textit{oblique.} See Figure 3. 
\begin{center}\label{Figure1}
	\includegraphics[width=80mm,scale=1.5]{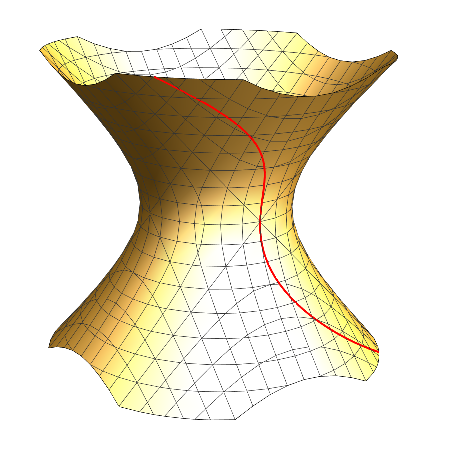}\\
	\small{Figure 3.}
\end{center}

\subsection*{Acknowledgements}
\normalfont
L. Grama is partially supported by 2018/13481-0;  2021/04003-0 (FAPESP) and 305036/2019-0 (CNPq).

\bibliographystyle{abbvr}

\section*{Bibliography}
\begin{thebibliography}{12}

\bibitem{arv1} D.Alekseevsky, A.Arvanitoyeorgos, Riemannian flag manifolds with homogeneous geodesics. Trans. Amer. Math. Soc. 359 (2007), no. 8, 3769--3789.

\bibitem{arv2} A. Arvanitoyeorgos, N.Souris, M.Statha, {\em Geodesic orbit metrics in a class of homogeneous bundles over real and complex Stiefel manifolds}. Geom. Dedicata 215 (2021), 31--50. 

\bibitem{Bes} L. A. Besse, \textit{Einstein manifolds. Reprint of the 1987 edition,} Classics in Mathematics (2008): 7.

\bibitem{BBGGS}E.Ballico, S.Barmeier, E.Gasparim, L.Grama, L.San Martin,\textit{A Lie theoretical construction of a Landau-Ginzburg model without projective mirrors}. Manuscripta Math. 158 (2019), no. 1--2, 85--101.

\bibitem{DZ} M. Djaa, A. Zagane, \textit{On geodesics of warped Sasaki metric,} Mathematical sciences and Applications E-Notes, Volume 5, Number 1 (2017), 85-92.

\bibitem{GGSM} E.Gasparim, L.Grama, L. San Martin, {\em Symplectic Lefschetz fibrations on adjoint orbits}. Forum Math. 28 (2016), no. 5, 967--979. 


\bibitem{EGSM} E. Gasparim, L. Grama and L. A. B. San Martin, \textit{Adjoint orbits of semi-simple Lie groups and Lagrangean submanifolds,} Proceeding of the Edinburgh Mathematical Society v. 60, p. 361-385, 2017.

\bibitem{GGN} B. Grajales, L. Grama and C. J. C. Negreiros, \textit{Geodesic Orbit Spaces in real flag manifolds,} to appear in Communications in Analysis and Geometry.

\bibitem{souris}  N.Souris, {\em Geodesic orbit metrics in compact homogeneous manifolds with equivalent isotropy submodules}. Transform. Groups 23 (2018), no. 4, 1149--1165. 
\end{thebibliography}
\renewcommand\refname{}

\end{document}